\documentclass[11pt]{amsart}

\usepackage[latin1]{inputenc}
\usepackage[english]{babel}
\usepackage{amsmath,amsthm,amssymb,epsfig,graphics,subfigure, wrapfig, epsfig, color}
\usepackage{mathtools}

\makeatletter
\@namedef{subjclassname@2010}{ \textup{2010} Mathematizcs Subject Classification}
\makeatother
\theoremstyle{lemma}
\newtheorem{theorem}{Theorem}
\theoremstyle{definition}
\newtheorem{question}{Question}
\theoremstyle{lemma}
\newtheorem{lemma}[theorem]{Lemma}
\theoremstyle{lemma}
\newtheorem{corollary}[theorem]{Corollary}
\theoremstyle{lemma}
\newtheorem{proposition}[theorem]{Proposition}
\theoremstyle{definition}

\theoremstyle{definition}
\newtheorem{remark}[theorem]{Remark}
\theoremstyle{definition}
\newtheorem{example}{Example}
\numberwithin{equation}{section}

\DeclareMathOperator{\supp}{supp}

\newcommand{\norm}[1]{\left\lVert#1\right\rVert}
\newcommand{\abs}[1]{\left\lvert#1\right\rvert}
\newcommand{\fun}[1]{\left\langle#1\right\rangle}

\begin{document}
\title{Weakly compact sets and weakly compact pointwise multipliers in Banach function lattices}

\author[Le{\'s}nik]{Karol Le{\'s}nik}

\author[Maligranda]{Lech Maligranda}

\author[Tomaszewski]{Jakub Tomaszewski}
\maketitle

\begin{abstract}
We prove that the class of Banach function lattices in which all relatively weakly compact sets are 
equi-integrable sets (i.e. spaces satisfying the Dunford--Pettis criterion) coincides with the class 
of $1$-disjointly homogeneous Banach lattices. A new examples of such spaces are provided. 
Furthermore, it is shown that Dunford--Pettis criterion is equivalent to de la Valle\'e Poussin 
criterion in all rearrangement invariant spaces on the interval. Finally, the results are applied 
to characterize weakly compact pointwise multipliers between Banach function lattices.  
\end{abstract}

\newcommand{\Addresses}{{
  \footnotesize
  K. Le{\'s}nik (Corresponding Author) and J. Tomaszewski, \textsc{Institute of Mathematics,
    Pozna{\'n} University of Technology, Poland, }
  \textit{e-mails:} \texttt{klesnik@vp.pl and jakub.tomaszewski@put.poznan.pl}

  L. Maligranda, \textsc{Department of Engineering Sciences and Mathematics, Lule\aa\ University\\
 of Technology, Sweden, }
  \textit{e-mail:} \texttt{lech.maligranda@ltu.se}
}}

\footnotetext[1]{2010 \textit{Mathematics Subject Classification}: Primary 46E30; Secondary 46B20, 
46B42, 46A50}

\footnotetext[2]{\textit{Key words and phrases}: Banach function spaces, rearrangement invariant 
spaces, Orlicz spaces, weakly compact sets, pointwise multipliers}
\footnotetext[3]{\Addresses}

\section{Introduction}

It is a classical question, whether one can describe all weakly compact sets in a given non-reflexive 
Banach space. The most well-known result in this direction is the Dunford--Pettis theorem, which 
says that relatively weakly compact sets in $L^1[0,1]$ are exactly the so-called equi-integrable 
sets (sometimes called uniformly-integrable) -- see \cite[pp. 376--378]{DP40} (see 
also \cite[Theorem 5.2.9]{AK06}, \cite[Theorem 5.2.9]{Di84} and \cite[Corollary 11, p. 294]{DS58}). 
This theorem has found many applications in analysis as well in probability theory. Therefore, not 
surprisingly that much attention was paid to deciding whether such kind of result can be extended 
to other separable non-reflexive Banach function spaces. 

In fact, it is independently interesting that Orlicz \cite{Or36} was the first who proved a theorem 
in this direction and he did so four years before Dunford and Pettis. Namely, in 1936 he proved 
that in the Orlicz space $L^{\varphi}[0,1]$ relatively weakly compact sets are exactly 
$L^{\varphi}$-equi-integrable sets, provided the function $\varphi$ is an $N$-function satisfying 
$\Delta_2$ and if the complementary function $\varphi^*$ to $\varphi$ satisfies the condition
\begin{equation} \label{1}
\lim_{u \rightarrow \infty} \frac{\varphi^*(2u)}{\varphi^*(u)} = \infty.
\end{equation}
These Orlicz spaces $L^{\varphi}[0,1]$ in a certain sense resemble the $L^1$-spaces but 
the assumption of being $N$-function exclude $L^1[0,1]$ space from Orlicz considerations. 
In 1978, Luxemburg \cite{Lu78} recalled Orlicz's result on relatively weakly compact sets 
in Orlicz spaces which at that time has not received the attention it deserves.

Later on, in 1994, Alexopoulos proved once again the Orlicz result, but he assumed a slightly 
weaker condition than (\ref{1}). Namely, he showed (cf. \cite[Corollary 2.9]{Al94}) that if an 
$N$-function $\varphi \in \Delta_2^{\infty}$ and its complementary function $\varphi^*$ 
satisfies 
\begin{equation} \label{2}
\lim_{u \rightarrow \infty} \frac{\varphi^*(\lambda u)}{\varphi^*(u)} = \infty ~{\rm for ~ some} ~\lambda > 1,
\end{equation}
then a bounded set $K \subset L^{\varphi}$ is relatively weakly compact if and only if
$K$ is $L^{\varphi}$-equi-integrable. 

The property (\ref{2}) appears to be crucial in the sequel, thus we say that an Orlicz function $\varphi$ 
satisfies the {\it $\Delta_0$-condition} and we write $\varphi \in \Delta_0$ if for some $\lambda > 1$ 
\begin{equation} \label{delta0}
\lim_{u \rightarrow \infty} \frac{\varphi (\lambda u)}{\varphi (u)} = \infty.
\end{equation}

Since Orlicz spaces are members of the wide class of Banach function lattices, a natural question 
arises, in which Banach function lattices does the analogous characteristics of relatively weakly 
compact sets remains valid? Such spaces are said to satisfy the {\it Dunford--Pettis criterion}, 
according to \cite{AKS08}. Quite recently, in 2008, the question was completely answered for 
rearrangement invariant space on $[0,1]$. Namely, Astashkin, Kalton and Sukochev in 
\cite[Theorem 5.5]{AKS08} proved that a rearrangement invariant space on $[0,1]$ satisfies 
the Dunford--Pettis criterion if and only if it has the property $(Wm)$, i.e. convergence in 
measure together with weak convergence implies convergence in the norm. Their method 
is based on techniques attributes to rearrangement invariant spaces and descriptions of 
weakly compact sets in such spaces from paper \cite{DSS01}.

Very recently Astashkin in \cite[Theorem 3.4]{As19a} found yet another characterization of 
rearrangement invariant spaces on $[0,1]$ satisfying the Dunford--Pettis criterion. Namely, 
he proved that this class coincides with the class of $1$-disjointly homogeneous spaces (cf. also 
\cite[Theorem 3.3]{As19a} and \cite[Proposition 4.9]{FHSTT14}). There is also result that 
$1$-disjointly homogeneous lattices are the same as the lattices with the positive Schur 
property (see \cite[Theorem 7]{Wn93} and \cite[Proposition 4.9]{FHSTT14}).

In the paper we extend Astashkin's characterization to the whole class of separable Banach function 
lattices (not only rearrangement invariant spaces) over arbitrary nonatomic measure space. Namely, 
our main result says that a separable Banach function lattice satisfies the Dunford--Pettis criterion 
if and only if it is $1$-disjointly homogeneous space. Our method is quite elementary and different 
than the one in \cite{As19a}.


The paper is organized as follows. In the next section we provide necessary definitions and discuss 
the notion of $X$-equi-integrable sets. The third section contains the main results. Firstly, we show 
that in Banach function spaces $X$-equi-integrable sets are not only relatively weakly compact, 
but even Banach--Saks sets. Next, we compare de la Vall\'ee Poussin condition with 
$X$-equi-integrability and show that those two notions are equivalent in the class of 
rearrangement invariant spaces on $[0,1]$, while de la Vall\'ee Poussin condition cannot 
be extended to spaces on $[0,\infty)$. This section is finished by the proof of the main 
result -- Theorem \ref{charakteryzacja}. In the fourth section we give new examples of $1$-disjointly 
homogeneous rearrangement invariant spaces on $[0,\infty)$. Moreover, we are able to characterize 
all $1$-disjointly homogeneous Orlicz spaces on $[0,\infty)$ in quite a constructive way (compare 
with another characterization from \cite[Theorem 5.1]{FHSTT14}). In the fifth section we comment 
the condition (\ref{2}) on Orlicz function and give a number of examples. Finally, in the last 
section we apply previous results to discuss weak compactness of pointwise multipliers. 

\section{Notation and preliminaries}
\subsection{Banach function lattices}
Let $(\Omega,\Sigma, \mu)$ be a $\sigma$-finite, complete and nonatomic measure space. 
As usual,  $L^{0}=L^{0}(\Omega)$ is the space of all (equivalence classes of) real-valued 
$\Sigma$-measurable functions defined on $\Omega$. A Banach space $X \subset L^0$ 
is said to be a {\it Banach function lattice} if:
\begin{itemize}
\item[$(i)$] $f\in L^{0}$, $g\in X$ and $|f|\leq |g|$ a.e. on $\Omega$ implies that $f\in X$
and $\Vert f\Vert \leq \Vert g\Vert $,
\item[$(ii)$] $X$ has a weak unit, i.e. an element $f\in X$ such that $f(t)>0$ for a.e. $t\in \Omega$.
\end{itemize}
 

A Banach function lattice $X$ is said to satisfy the {\it Fatou property} if for each sequence 
$(f_n)\subset X$ satisfying $f_n\uparrow f$ $\mu$-a.e. on $\Omega$ and 
$\sup_{n \in \mathbb N}\|f_n\|_X<\infty$, there holds $f\in X$ and $\norm{f}_X\leq \sup_n\norm{f_n}_X$. 
 
An element $f\in X$  is said to be {\it order continuous} if for any sequence $(f_n)\subset X$ with 
$0 \leq f_n \leq |f|$ and $f_n \rightarrow 0$ $\mu$-a.e. on $\Omega$ there holds $\norm{f_n}_{X} \rightarrow 0$.
By $X_{a}$ we denote the subspace of all order continuous elements of $X$. A Banach function space $X$ 
is called {\it order continuous} (we write $X\in (OC)$ for short) if $X_a = X$. It will be used few times 
in the sequel that $f\in X_{a}$ if and only if $\norm{f\chi _{A_{n}}}_{X} \rightarrow 0$ for any sequence 
$(A_n)$ satisfying $A_n \downarrow \emptyset$,  where $A_n \downarrow \emptyset$ means that 
$(A_n)$ is decreasing sequence of measurable sets with the intersection of measure zero 
(see \cite[Proposition 3.5, p. 15]{BS88}). The subspace $X_a$ is always closed in $X$ 
(see \cite[Th. 3.8, p. 16]{BS88}).

with equal sets they norms must be equivalent. 

An important class of Banach function lattices constitute rearrangement invariant spaces. 
Consider $I=[0,\alpha)$, where $0<\alpha\leq \infty$ with the Lebesgue measure $m$. 
Recall that the distribution function of $f\in L^0(I)$ is defined by
$$
d_f(\lambda) = m(\{t \in I : |f(t)| > \lambda\})
$$
for $\lambda \geq 0$. We say that two functions  $f, g \in L^{0}(I)$  are equimeasurable when 
they have the same distribution functions, i.e. $d_f \equiv d_y$. Then a Banach function lattice 
$X$ on $I$ is called {\it rearrangement invariant} (or {\it symmetric}) if for two given equimeasurable 
functions $f, g \in L^{0}(I)$ with $f\in X$ there holds $g \in X$ and $\norm{f}_{X} = \norm{g}_{X}$. 
In particular, $\norm{f}_{X} = \norm{f^{*}}_{X}$, where $f^{*}$ is the non-decreasing  rearrangement 
of $f$, i.e.,
$$
f^{*}(t) := \inf\{\lambda > 0 : d_f(\lambda) \leq t \}
$$
for $t \geq 0$. 
For more information on rearrangement invariant spaces we refer to books \cite{BS88}, \cite{KPS82} 
and \cite{LT79}.

If not specified otherwise, we will understand that all general Banach function lattices are defined on 
an arbitrary $\sigma$-finite, complete and nonatomic measure space $(\Omega,\Sigma, \mu)$, while 
all rearrangement invariant spaces are on $I=[0,1]$ or $I=[0,\infty)$ with the Lebesgue measure $m$. 

\subsection{Orlicz functions and Orlicz spaces}
In this this paper $\varphi$ will always denote an {\it Orlicz function}, i.e. a  continuous increasing 
and convex function $\varphi \colon [0,\infty)\to [0,\infty)$ such that $\varphi(0)=0$. We will assume 
that Orlicz function $\varphi$ is {\it coercive}, i.e. it has the additional property of 
$\lim_{u \rightarrow \infty} \frac{\varphi(u)}{u} = \infty$, or is an {\it N-function} if additionally
\begin{equation} \label{N-function}
\lim_{u \rightarrow 0} \frac{\varphi(u)}{u} = 0 ~~{\rm and} ~~\lim_{u \rightarrow \infty} \frac{\varphi(u)}{u} = \infty.
\end{equation}
Property of coerciveness of $\varphi$ exclude the case of $\varphi(u) = a u$, but ensures that the conjugate 
function has finite values. For an Orlicz function $\varphi$ the conjugate function $\varphi^*$ is defined by
\begin{equation*} 
\varphi^*(v)=\sup_{u \geq 0} \, [uv - \varphi(u)], ~ v \geq 0.
\end{equation*} 
Then $\varphi^*$ is finite-valued if and only if $\varphi$ is coercive. Moreover, if $\varphi$ is an N-function 
then $\varphi^*$ is also an N-function. To avoid pathologies through the paper, we will understand that all 
Orlicz functions are coercive.

We say that an Orlicz function $\varphi$ satisfies the {\it $\Delta_2$-condition for large u} 
(or {\it at infinity}) and we write $\varphi \in \Delta_2^{\infty}$ if 
$\limsup_{u \rightarrow \infty} \frac{\varphi (2u)}{\varphi (u)} < \infty$ or, equivalently, there 
exist constants $C > 1$ and $u_0 \geq 0$ such that
\begin{equation} \label{delta2}
\varphi (2u) \leq C \varphi (u) ~~ {\rm for ~all} ~u \geq u_0.
\end{equation}  
If condition (\ref{delta2}) holds with $u_0 = 0$, then we say that an Orlicz function 
$\varphi$ satisfies the {\it $\Delta_2$-condition for all u} and we write $\varphi \in \Delta_2^a$.

Further definitions, properties and results about Orlicz's functions or N-functions are taken from the 
books \cite{KR61}, \cite{Ma89} and \cite{RR91}.

The {\it Orlicz space} $L^{\varphi} = L^{\varphi}(\Omega)$ on a $\sigma$-finite complete nonatomic 
measure space $(\Omega, \mu)$ is the space of all $f\in L^0(\Omega)$ satisfying
$I_{\varphi}(\lambda f) < \infty {\rm\ for\ some}\ \lambda = \lambda(f) > 0$, where the modular $I_\varphi$ 
is given by
$$
I_{\varphi}(f):= \int_{\Omega} \varphi (|f(x)|) \, d\mu(x).
$$
This space is a Banach function lattice with the {\it Luxemburg--Nakano norm} 
defined as
$$
\norm{f}_{\varphi}:= \inf\{\lambda>0 \colon I_{\varphi}( f/\lambda) \leq 1\}.
$$
For $\Omega = I$ with the Lebesgue measure $m$,  
the Orlicz space $L^{\varphi}(I)$ is a rearrangement-invariant space 
with the Fatou property (see \cite{BS88} and \cite{KPS82}).
\vspace{2mm}

\subsection{Equi-integrable sets}
Let us recall that, classically, a bounded subset $K$ of $L^1[0,1]$ is called {\it equi-integrable} 
(or {\it uniformly integrable}) when for each $\varepsilon > 0$ there is $\delta = \delta(\varepsilon) >0$ 
such that for every set $A \subset [0, 1]$ with $m(A)<\delta$ we have
$$
\sup_{f\in K} \int_A |f| \, dm = \sup_{f\in K} \norm{f\chi_A}_1 < \varepsilon, 
$$
i.e., $\lim_{m(A) \rightarrow 0} \sup_{f\in K}  \int_A |f| \, dm = 0$.
This notion generalizes easily to an arbitrary Banach function lattice on a finite measure 
space -- it is enough to replace the norm $\|\cdot\|_1$ by an abstract norm $\|\cdot\|_X$. However, it 
is not the right definition when we wish to deal with an infinite measure spaces as well. 
Therefore, we define $X$-equi-integrable sets in a slightly different way. 

Let $X$ be a Banach function lattice and $K\subset X$ be a bounded set. We say that a set 
$K$ is {\it $X$-equi-integrable} (or $K$ has {\it equi-absolutely continuous norms in $X$}) 
if for each sequence of measurable sets $(A_n)\subset \Omega$ such that 
$A_n\downarrow\emptyset$, there holds 
$$
 \sup\limits_{f\in K}\norm{f\chi_{A_n}}_X\rightarrow 0 \text{ when } n\rightarrow\infty.
$$
The following trivial lemma ensures that in case of finite measure our definition of $X$-equi-integrability 
coincides with another ones. 

\begin{lemma}\label{ei-rownowazne}
{\it Let $X$ be a Banach function lattice on a finite measure space $(\Omega, \mu)$ and let $K\subset X$. 
The following statements are equivalent:
\begin{itemize}
\item[(i)] For each $\varepsilon > 0$ there is $\delta>0$ such that for every $A \subset \Omega$ with 
$\mu(A)<\delta$ we have 
$$
\sup_{f\in K}\norm{f\chi_A}_X < \varepsilon. 
$$
\item[(ii)] For each sequence of measurable sets $(A_n), A_n \subset \Omega$ such that $\mu(A_n) \to 0$ 
there holds 
$$
 \sup\limits_{f\in K}\norm{f\chi_{A_n}}_X\rightarrow 0 \text{ when } n\rightarrow\infty.
$$
\item[(iii)] For each sequence of measurable sets $(A_n)$, $A_n \subset \Omega$ such that $A_n\downarrow\emptyset$ 
there holds 
$$
 \sup\limits_{f\in K}\norm{f\chi_{A_n}}_X\rightarrow 0 \text{ when } n\rightarrow\infty.
$$
\end{itemize}
}
\end{lemma}

\proof
Evidently, $(i)\implies (ii)\implies (iii)$, thus we will explain only $(iii)\implies (i)$. 
Suppose there is $\varepsilon>0$ such that for each $n\in \mathbb{N}$ there is $B_n \subset \Omega$ 
with $\mu(B_n)<1/2^n$ satisfying
$$
\sup\limits_{f\in K}\norm{f\chi_{B_n}}_X> \varepsilon.
$$
It is enough to take $A_n=\bigcup_{k=n}^{\infty}B_k$, to see that $(iii)$ does not hold. 
\endproof

In the sequel we will need the following technical lemma generalizing Lemma 2.6 from \cite{Al94}. 

\begin{lemma}\label{zbiory} 
Let $X$ be an order continuous Banach function lattice on a measure space $(\Omega, \mu)$. 
If a bounded set $K\subset X$ is not $X$-equi-integrable, then there exists a sequence $(f_n)\subset K$, 
a number $\varepsilon>0$ and a sequence of disjoint measurable sets $(A_n), A_n \subset \Omega$ 
such that for every $n \in \mathbb N$ 
$$
\norm{f_n\chi_{A_n}}_X > \varepsilon.
$$
\end{lemma}

\proof
When $K\subset X$ is not $X$-equi-integrable then for some $ \varepsilon>0$ there exist a sequence 
of sets $(B_n)$, $B_n\downarrow\emptyset$, and a sequence $(f_n)\subset K$ such that 
$$
\|f_n\chi_{B_n}\|_X\geq  \varepsilon.
$$
Choose $n_1=1$. Since $f_{n_1}\in X_a=X$ we can find $n_2$ such big that
\[
\norm{f_{n_1}\chi_{B_{n_1}\setminus B_{n_2}}}_X\geq \varepsilon/2.
\]
We proceed in the same way with $n_2$ in place of $n_1$. In consequence, we get a sequence $(n_k)$ 
such that sets $A_k=B_{n_k}\setminus B_{n_{k+1}}$ and functions $f_{n_k}$ satisfy the thesis with 
$\varepsilon:= \varepsilon/2$. 
\endproof

It is known \cite[Proposition 3.6.5.]{MN91} that in any Banach lattice $X$ every $X$-equi-integrable 
set is relatively weakly compact (see \cite{MN91} for notion of $X$-equi-integrable sets in abstract 
lattices). The question is whether the opposite implication also holds, or rather, in which spaces 
the opposite implication also holds. Following \cite{AKS08} we will say that a Banach function 
lattice $X$ satisfies {\it the Dunford--Pettis criterion} when each relatively weakly compact set in $X$ 
is $X$-equi-integrable.

Trivially, none of $L^p$ with  $1<p<\infty$ satisfies the Dunford--Pettis criterion. 
Moreover, it cannot hold also in any Banach function lattice which is not order continuous 
(take just $K=\{f\}$ for any $f\in X\setminus X_a$). In particular, if  $K\subset X$ is 
$X$-equi-integrable then $K\subset X_a$.


\subsection{$1$-disjointly homogeneous spaces and the positive Schur property}\label{2.4}

Let $X$ be a Banach function lattice. We say that $X$ is {\it $1$-disjointly homogeneous space} 
(1-DH for short) if every normalized sequence $(f_n)\subset X$ of disjointly supported functions 
has a subsequence  $(f_{n_k})$ equivalent to the $\ell^1$-basis, i.e. there exists a constant $c>0$ 
such that for every $a=(a_k)\in\ell^1$ 
$$
c\norm{a}_1\leq \norm{\sum\limits_{k=0}^{\infty}a_kf_{n_k}}_X\leq  \norm{a}_1.
$$
Evidently, if we required only that the sequence $(\norm{f_n}_X)$ is just bounded from above and 
from below, we would get the equivalent definition. 


$1$-disjointly homogeneous spaces and, more generally, $p$-disjointly homogeneous spaces 
($1 \leq p < \infty$), have been intensively investigated during the last decade -- see 
\cite{As15, As19b, FHST12, FHSTT14, FHT16}.

Examples of 1-DH Banach lattices are: the Orlicz spaces $L^{\varphi}[0, 1]$ 
when $\varphi \in \Delta_2^{\infty}$ and $\varphi^* \in \Delta_0$ (cf. \cite {Le89}, 
see also \cite{FHST12}) and the Lorentz spaces $\Lambda_w [0, 1]$ with the norms 
$\|f\|_w = \int_0^1 f^*(t) w(t)\, dt$, where $w$ is a positive, nonincreasing function on 
$(0, 1]$, such that $\lim\limits_{t \rightarrow 0^+} w(t) = \infty, w(1) > 0$ and $\int_0^1 w(t)\, dt =1$ 
(cf. \cite[Theorem 5.1]{FJT75}). In particular, for $w_p(t) = \frac{1}{p} t^{1/p-1}, 1 \leq p < \infty$, 
we have $\Lambda_{w_p}[0, 1] = L^{p, 1}[0, 1]$, which are 1-DH spaces.  

Let us also mention that Banach lattices being 1-DH have previously been considered under a different approach. 
Wnuk \cite{Wn89} started to investigate the positive Schur property: a Banach lattice $X$ 
has the {\it positive Schur property} if every weakly null sequence with positive terms is norm 
convergent, see also \cite{Wn89, Wn93}. It follows from, e.g. \cite[Corollary 2.3.5]{MN91} 
that it is suffices to verify this condition for disjoint sequences. Using Rosenthal's $l^1$-theorem, 
Wnuk proved in \cite[Theorem 7]{Wn93} that a Banach lattice $X$ is 1-DH if and only if $X$ 
has the positive Schur property (see also \cite[Proposition 4.9]{FHSTT14} and \cite[Proposition 1.2]{Ra91}). 
Note that the positive Schur property (as well as 1-DH) is not preserved by isomorpisms (see \cite[p. 18]{Wn93}). Surprisingly, in \cite{As19c}, it was proved that in the case of rearrangement invariant
spaces on $[0,1]$ the situation is completely different (see \cite[Theorem 5 and Corollary 2]{As19c}).

\section{Weakly compact sets in Banach function lattices}
As we mentioned in the previous section, each $X$-equi-integrable set is relatively weakly compact 
even in abstract lattices. For subsets of Banach function lattices we can prove more, i.e. that each 
equi-integrable set is also a Banach--Saks set. In the case of Orlicz spaces on $[0,1]$ such 
theorem was proved in \cite{Al94}, while for order continuous rearrangement invariant spaces on 
$[0,\infty)$ it is the statement of \cite[Theorem 4.10]{DSS04}.

Recall, that given a Banach space $X$ and a set $K\subset X$ we call $K$ the {\it Banach--Saks set} 
if for each sequence $(f_n)\subset K$ there exists $f\in X$ and a subsequence $(f_{n_k})$ such that 
every further subsequence $(f_{n_{k_l}})$ has means norm converging to $f$, i.e. 
$(\frac{1}{n}\sum\limits_{l=1}^n f_{n_{k_l}})$ converges to $f$ in norm.
Note that each Banach--Saks set is relatively weakly compact (see, for example, \cite[Proposition 2.3]{LRT14}). 

\begin{theorem}
Let $X$ be a Banach function lattice with the Fatou property such that 
$L^1 \cap L^{\infty} \subset X$. If $K\subset X$ is a bounded and $X$-equi-integrable set, 
then $K$ is a Banach--Saks set in $X$.
\end{theorem}
\begin{proof}
Let $K\subset X$ be an $X$-equi-integrable set and let $(f_n)\subset K$. By the Day--Lennard 
theorem (cf. \cite[Theorem 3.1]{DL10}) there exists a function $f\in X$ and a subsequence $(f_{n_k})$ 
such that for each further subsequence $(f_{n_{k_j}})$
$$
\frac{1}{n}\sum\limits_{j=1}^n f_{n_{k_j}}(t)\rightarrow f(t)\text{ for a.e. } t\in \Omega,
$$
as $n\rightarrow\infty$. Firstly, we will show that $f\in X_a$. Let $h\in X'$ with $\norm{h}_{X'}\leq 1$ 
and $A\subset \Omega$ be a measurable set. We have
\begin{align*}
\abs{\int_{ \Omega} h\chi_A fd\mu}&\leq\int_{ \Omega}\abs{h\chi_A f}d\mu
\leq \liminf_{n \rightarrow \infty}\int_{ \Omega}\abs{h\chi_A \frac{1}{n}\sum\limits_{k=1}^n f_{n_k}}d\mu\\
&\leq\sup\limits_{n \in \mathbb N} \frac{1}{n}\sum\limits_{k=1}^n\int\abs{h\chi_A f_{n_k}}d\mu
\leq\sup\limits_{n \in \mathbb N} \frac{1}{n}\sum\limits_{k=1}^n\norm{f_{n_k}\chi_A}_X\norm{h}_{X'}\\
&\leq \sup\limits_{n \in \mathbb N} \frac{1}{n}\sum\limits_{k=1}^n\sup\limits_{g\in K}\norm{g\chi_{A}}_X 
=\sup\limits_{g\in K}\norm{g\chi_{A}}_X.
\end{align*}
Thus, by the Fatou property of $X$,  $\norm{f\chi_A}_X\leq\sup\limits_{g\in K}\norm{g\chi_{A}}_X$ 
for arbitrary $A\subset \Omega$.  Since $K$ is $X$-equi-integrable set, we conclude that $f\in X_a$. 

Let $(f_{n_{k_i}})$ be a subsequence of $(f_{n_k})$. We will show that means of $(f_{n_{k_i}})$ are norm 
convergent to $f$. Denote 
$$
g_n=\frac{1}{n}\sum\limits_{i=1}^n f_{n_{k_i}}.
$$
Let $\varepsilon>0$ be arbitrary. Consider the sequence $B_n=\bigcup_{k>n}\Omega_k\downarrow\emptyset$ 
with $n \to \infty$, where the sequence $(\Omega_n)$ comes from the definition of $\sigma$-finite measure 
space. By $X$-equi-integrability of $K$ and order continuity of $f$ there is $i\in\mathbb{N}$ such that 
 $$
 \norm{f\chi_{B_i}}_X\leq \varepsilon {\rm\ and \ }\sup\limits_{g\in K}\norm{g\chi_{B_i}}_X\leq \varepsilon.
 $$
Denote $B:=B_i$ and notice that for $B'=\Omega \setminus B$ there holds $\mu(B')<\infty$, by definition of $B_n$ sets.
Moreover, by Lemma \ref{ei-rownowazne}, there is $\delta>0$ such that  
$$
\norm{f\chi_{A}}_X\leq \varepsilon {\rm\ and \ } \sup\limits_{g\in K}\norm{g\chi_{A}}_X\leq \varepsilon,
$$
whenever $\mu(A)<\delta$ and $A\subset B'$. In consequence, also 
$$
 \norm{g_n\chi_{A}}_X\leq \varepsilon {\rm\ and \ }  \norm{g_n\chi_{B}}_X\leq \varepsilon,
$$
for each $n \in \mathbb N$ and each $A$ like above. 
On the other hand, by the Egorov theorem there exists $C\subset B'$ such that $\mu(B'\setminus C)\leq\delta$ 
and $g_n$ converges uniformly to $f$ on $C$. Let $N\in\mathbb{N}$ be such 
 $$
 \norm{(f-g_n)\chi_{C}}_{\infty}\leq\frac{\varepsilon}{\norm{\chi_{C}}_X}
 $$  
 for $n\geq N$. Therefore, 
\begin{align*}
\norm{f-g_n}_X&\leq \norm{(f-g_n)\chi_{C}}_X+\norm{(f-g_n)\chi_{B'\setminus C}}_X+\norm{(f-g_n)\chi_B}_X\\
&\leq\frac{\varepsilon}{\norm{\chi_{C}}_X}\norm{\chi_{C}}_{X}+\norm{f\chi_{B'\setminus C}}_X+
\norm{g_n\chi_{B'\setminus C}}_X+\norm{f\chi_B}_X+\norm{g_n\chi_B}_X\\
&\leq 5 \varepsilon 
\end{align*}
for $n\geq N$, which finishes the proof.
\end{proof}

While the Dunford--Pettis description of relatively weakly compact sets in terms of equi-integrability 
attracted considerable attention and the notion of equi-integrable sets found its conterparts in general 
function lattices, the alternative description by de la Vall\'ee Poussin seems to be much less popular. On the 
other hand, some counterparts of de la Vall\'ee Poussin theorem appear in the subject of strong embeddings 
of rearrangement invariant spaces (see, for example, \cite[Lemma 4]{As14} and references therein).
The following theorem shows that de la Vall\'ee Poussin characterization works in all  rearrangement invariant 
spaces on $[0,1]$. 

\begin{theorem}
Let $X$ be a rearrangement invariant space on $I=[0,1]$ such that $X_a\neq\{0\}$ and 
let a set $K\subset X$ be bounded. The following statements are equivalent:
 \begin{enumerate}
 \item $K$ is $X$-equi-integrable.
 \item There holds 
$$
\lim\limits_{\gamma \rightarrow \infty} \sup\limits_{f\in K}\norm{f\chi_{\{\abs{f} > \gamma\}}}_X=0.
$$
\item There exists an increasing convex function $\varphi \colon [0, \infty) \rightarrow [0, \infty)$ 
with $\varphi(0) = 0$ satisfying $\lim\limits_{u\rightarrow\infty}\frac{\varphi(u)}{u}=\infty$
such that $ \sup\limits_{f\in K}\norm{\varphi(f)}_X<\infty$.
\end{enumerate}  
\end{theorem}

\begin{proof}
$(i)\Rightarrow(ii)$. Let $\varepsilon>0$. Choose $\delta>0$ such that  
$$
\sup\limits_{f\in K}\norm{f\chi_{A}}_X\leq\varepsilon
$$
when $m(A)\leq\delta$. Since $X$ is rearrangement invariant space, thus $X\xhookrightarrow{C}L^1$ 
for some constant $C\geq 1$ (cf. \cite{BS88} or \cite{KPS82}). By the Chebyshev inequality
$$
m(\{\abs{f}> \gamma\}) \leq \frac{\norm{f}_1}{\gamma} \leq C\frac{\norm{f}_X}{\gamma}\leq\delta
$$
for each $f\in K$ and $\gamma $ large enough. Consequently 
$$
\lim\limits_{\gamma \rightarrow\infty}\sup\limits_{f\in K}\norm{f\chi_{\{\abs{f}>\gamma \}}}_X=0.
$$ 
$(ii)\Rightarrow (i)$. 
 Let $\varepsilon>0$. By the assumption there exists $\gamma_0$ such that
 $$
 \sup\limits_{f\in K}\norm{f\chi_{\{\abs{f}>\gamma_0\}}}_X\leq \varepsilon.
 $$
 Let $(A_n)$ be a sequence of measurable sets such that $A_n\downarrow\emptyset$. Since 
 $X_a\neq\{0\}$, we can choose $N\in\mathbb{N}$ such that 
 $$
 f_X(m(A_n))<\frac{\varepsilon}{\gamma_0}
 $$
for $n\geq N$. In consequence,
\begin{align*}
\norm{f\chi_{A_n}}_X&\leq \norm{f\chi_{\{\abs{f}<\gamma_0\}\cap A_n}}_X+\norm{f\chi_{\{\abs{f}>\gamma_0\}\cap A_n}}_X\\
&\leq \gamma_0\norm{\chi_{A_n}}_E+\varepsilon =2 \varepsilon,
\end{align*}
for $n\geq N$.
It means that $K$ is $X$-equi-integrable set. 
\vspace{2mm}

$(iii)\Rightarrow (ii)$. Let $\varepsilon>0$. By the assumption there exists $u_{\varepsilon}>0$ such that 
$$
u \leq \varepsilon \, \varphi(u)
$$
for all $u\geq u_{\varepsilon}$.
Denote $M:=\sup\limits_{f\in K}\norm{\varphi(f)}_X$. We have for $f\in K$
$$
\| f\chi_{\{\abs{f}>u_{\varepsilon}\}}\|_X \leq \varepsilon \, \| \varphi(\abs{f}) 
\chi_{\{\abs{f}>u_{\varepsilon}\}}\|_X \leq \varepsilon \sup\limits_{f\in K} \| \varphi(f)\|_X \leq M \varepsilon.
$$

$(ii)\Rightarrow (iii)$. Let $u_1=0$. For each $n\geq 2$ we can choose $u_n>0$ in such a way that
$$
\sup\limits_{f\in K}\norm{f\chi_{\{\abs{f}>u_n\}}}_X\leq \frac{1}{n^2}.
$$
and $u_{n+1}\geq 2u_n$ for each $n\geq 2$.
For $u\geq 0$ define a function 
$$
\varphi(u)=\sum\limits_{n=1}^\infty (u-u_n)_+.
$$
Clearly, $\varphi$ is an increasing convex function with $\varphi (0) = 0$. Let us explain that  also 
$\lim\limits_{u\rightarrow\infty}\frac{\varphi(u)}{u}=\infty$.  
In fact, for $u\in [u_n,u_{n+1})$ there holds 
\[
\varphi(u)=\sum_{k=1}^n (u-u_k)_+=nu-\sum_{k=1}^n u_k\geq nu-2 u_n,
\]
which means that $\frac{\varphi(u)}{u}\geq n-2\frac{u_n}{u}\geq n-2$ and proves the claim. 
It remains to see that $\sup\limits_{f\in K}\norm{\varphi(f)}_X<\infty$.
For $f\in K$ we have $\varphi(|f|)\leq \sum\limits_{n=1}^\infty |f | \chi_{\{\abs{f}>u_n\}}$,
thus
$$
\norm{\varphi(|f|)}_X\leq\sum\limits_{n=1}^\infty\norm{f\chi_{\{\abs{f}>u_n\}}}_X\leq \frac{\pi^2}{6},
$$
and the proof is finished.
\end{proof}

\begin{remark}
Observe that one cannot get similar characterization of $X$-equi-integrable sets in rearrangement 
invariant spaces on $I=[0,\infty)$. Indeed, if $\varphi>0$ is an arbitrary positive function defined 
on $[0,\infty)$, then the set
$$
K=\{\chi_{[n,n+1)} \colon n\in\mathbb{N}\}\subset X
$$ 
is not $X$-equi-integrable, while we have  
$$
\sup\limits_{f\in K}\norm{\varphi(f)}_X=\varphi(1)\norm{\chi_{[0,1)}}_X<\infty.
$$
It means that de la Vall\'ee Poussin condition cannot imply equi-integrability in the case of infinite 
measure space. 
\end{remark}

To prove our main result we will need the following Rosenthal's type lemma (cf. \cite[p. 82]{Di84}) 
proved by Alexopoulos in \cite[Lemma 2.7]{Al94}. 

\begin{lemma} \label{alexorosen} 
Let $X$ be a Banach space, $(x_n)\subset X$ be a weakly null sequence and $(x^*_n)\subset X^*$ 
be a weakly* null sequence. For every $\varepsilon>0$ there exists an increasing sequence 
$(k_i)\subset \mathbb{N}$ such that 
$$
\sum\limits_{j\neq i}\abs{\fun{x_{n_i},x^*_{n_j}}}<\varepsilon ~ {\it for ~ each} ~ i\in\mathbb{N}.
$$
\end{lemma}

\begin{theorem}\label{charakteryzacja}
{\it Let $X$ be a Banach function lattice. Then $X$ satisfies the Dunford--Pettis criterion 
if and only if $X$ is $1$-disjointly homogeneous.}
\end{theorem}

\begin{proof}
\textit{Sufficiency}. Assume that $X$ is not $1$-disjointly homogeneous, i.e. there exists a normalized 
sequence $(f_n)\subset X$ of disjointly supported function without subsequence equivalent to the $\ell^1$ 
basis. By the Roshental's $\ell^1$-theorem, the sequence $(f_n)$ contains weakly Cauchy subsequence 
$(f_{n_k})$. We will show that $(f_{n_k})$ is weakly null. Suppose for a contradiction, that there exist 
$g\in X'$ and $a\neq 0$ such that
$$
\lim\limits_{k\rightarrow\infty}\int_{\Omega} gf_{n_k} \, d\mu = a.
$$
Denote $A=\bigcup\limits_{k=0}^{\infty}\supp(f_{n_{2k}})$. Observe that evidently $h=g\chi_{A}\in X'$. 
We have 
$$
\lim\limits_{k\rightarrow\infty}\int_{\Omega} hf_{n_{2k}}d\mu=a {\rm \ while\ } \int_{\Omega} hf_{n_{2k+1}}d\mu=0, 
$$
which means that the sequence $(\int_{\Omega} hf_{n_k}d\mu)$ is not convergent. Thus, we arrived 
to the contradiction with $(f_{n_k})$ being weakly Cauchy. Therefore, $(f_{n_k})$ is weakly null and, 
in particular, the set $K=\{f_{n_k}\}$ is weakly compact. The proof of sufficiency will be finished when 
we prove that $K$ is not $X$-equi-integrable set. Let 
$$
A_i=\bigcup\limits_{k=i}^{\infty}\supp(f_{n_k}).
$$
Clearly $A_i\downarrow\emptyset$. However, 
$$
\sup\limits_{f\in K}\norm{f\chi_{A_i}}_X\geq\norm{f_{n_i}\chi_{A_i}}_X=\norm{f_{n_i}}_X= 1.
$$
Thus $X$ does not satisfy the Dunford--Pettis criterion.
\vspace{2mm}

\textit{Necessity}. Let $K\subset X$ be an arbitrary relatively weakly compact set and assume it is not 
$X$-equi-integrable. One can choose $(g_m)\subset K$ which is not $X$-equi-integrable. Since $K$ is 
weakly compact there exists a subsequence $(g_{m_n})$ converging weakly to some $g\in X$. Define 
$$
f_n=g_{m_n} - g.
$$
Then $(f_n)$ is weakly null. Moreover, it is also not $X$-equi-integrable set, since $g\in X_a=X$.
By Lemma \ref{zbiory} there exist  $\varepsilon>0$ and a sequence of disjoint measurable sets $(A_k)$ 
such that for some subsequence $(f_{n_k})$ of $(f_n)$
$$
\norm{f_{n_k}\chi_{A_k}}_X>\varepsilon
$$
for each $n\in\mathbb{N}$. Since $X$ is  $1$-disjointly homogeneous  there exists a further subsequence 
$(f_{n_{k_l}}\chi_{A_{k_l}})$ of $(f_{n_k}\chi_{A_k})$ equivalent to the  $\ell^1$ basis. Denote 
$w_l:=f_{n_{k_l}}\chi_{A_{k_l}}$.
We claim that there is  $h\in X'$ such that 
$$
\fun{\sum\limits_{l=0}^\infty a_lw_l,h}=\sum\limits_{l=0}^\infty a_l
$$
for each $a\in\ell^1$. Indeed, let $U=[w_l]$ be a closed span of $(w_l)$ and $T\colon U\rightarrow\ell^1$ 
be an isomorphism transforming $(w_l)$ into $\ell^1$ basis, i.e. $T(w_l)=e_l$, for $l\in \mathbb{N}$. 
Let $y^*\in \ell^\infty$ be such that $\fun{a,y^*}=\sum\limits_{l=0}^\infty a_l$ for each $a\in\ell^1$. 
Denote $h^*=T^*y^*\in U^*$. By the Hahn-Banach  theorem there exists $h\in X'=X^*$ such that 
$h_{\mid U}=h^*$. Clearly $\fun{\sum\limits_{l=0}^\infty a_lw_l,h}=\sum\limits_{l=0}^\infty a_l$ for each 
$a\in\ell^1$ and the claim follows.

Denote $h_l=h\chi_{A_{k_l}}$ and observe that for each $l\in \mathbb{N}$
$$
\fun{ f_{n_{k_l}},h_l}=\int f_{n_{k_l}}h_ld\mu=\int f_{n_{k_l}}h\chi_{A_{k_l}}d\mu=\fun{ w_l,h}=1.
$$
Moreover,  $(h_l)$ is weak*-null sequence. Indeed, for each $f\in X$ we have
$$
\abs{\int_{\Omega} fh_ld\mu} = \abs{\int_{A_{k_l}} fh_ld\mu} 
\leq \norm{f\chi_{A_{{k_l}}}}_X\norm{h_l}_{X'}
\leq \norm{f\chi_{A_{{k_l}}}}_X\norm{h}_{X'}\rightarrow 0,
$$
since $X$ is order continuous space, as is each 1-DH space.

By Lemma \ref{alexorosen} there exists an increasing sequence $(l_j)\subset \mathbb{N}$ such 
that for $j\in\mathbb{N}$ 
$$
\sum\limits_{j\neq j}\abs{\fun{f_{n_{k_{l_j}}},h_{l_i}}}<\frac{1}{2}.
$$
We have for each $i \in \mathbb N$
\begin{align*}
\abs{\fun{f_{n_{k_{l_i}}},h}}&=\abs{\fun{ f_{n_{k_{l_i}}},\sum\limits_{j=1}^{\infty}h_{l_j}}}\\
&\geq \abs{\fun{ f_{n_{k_{l_i}}},h_{l_i}}}-\sum\limits_{j\neq i}\abs{\fun{f_{n_{k_{l_i}}},h_{l_j}}} \geq 1-\frac{1}{2}=\frac{1}{2}.
\end{align*}
However, $(f_n)$ was weakly null, so we get a contradiction, which finishes the proof.
\end{proof}

\section{$1$-disjointly homogeneous spaces on $[0,\infty)$}
The main known examples of $1$-disjointly homogeneous spaces are Lorentz $L^{p,1}$ spaces 
on $I = [0,1]$ (generated by the norm $\| f\|_X = \frac{1}{p} \int_I x^{1/p-1} f^*(x) \, dx$ for $1 < p < \infty$)
and Orlicz spaces $L^{\varphi}$ spaces on $[0,1]$, when $\varphi \in \Delta_2^{\infty}$ and 
its complementary function $\varphi^*$ satisfies condition (\ref{2}).

It seems however, that much less is known about such spaces on semiaxis. 
The following theorems provide such examples. 


\begin{theorem}
For each $1 < p < \infty$ the space $(L^{p,1}\cap L^1)[0,\infty)$ is $1$-disjointly homogeneous.
\end{theorem}

\proof
To simplify the notion let us denote $E=(L^{p,1}\cap L^1)[0,\infty)$. Let $(f_n)\subset (L^{p,1}\cap L^1)[0,\infty)$ 
be a normalized sequence of positive functions with disjoint supports. We need to consider two cases. \\
$1^0.$ Firstly, if there is $\delta>0$ such that $\norm{f_n}_1>\delta$ for infinitely many $n$'s, then 
composing these $n$'s into a sequence $(n_k)$ we have for each $a=(a_k)\in l^1$
$$
\sum_{k=1}^{\infty}|a_k|\geq \norm{\sum_{k=1}^{\infty}a_k f_{n_k}}_E
\geq \norm{\sum_{k=1}^{\infty}a_k f_{n_k}}_1
\geq \delta \sum_{k=1}^{\infty}|a_k|.
$$\\
$2^0.$ Otherwise $\norm{f_n}_1\to 0$ and thus $\norm{f_n}_{p,1}=1$ for almost all $n$'s. We may assume 
that actually $\norm{f_n}_{p,1}=1$ for all $n$'s. We claim that for each $k\in \mathbb{N}$ there is 
$n\in \mathbb{N}$ such that
\begin{equation}\label{star}
\norm{f_n\chi_{\{f_n>k\}}}_{p,1}\geq 1/2.
\end{equation}
Suppose for the moment that we have proved the claim. Then the claim together with 
$\norm{f_n}_{p,1}=1$ imply that there is an increasing sequence $(n_k)$ such that 
$$
m(\{f_{n_k}>k\})\to 0 \ {\rm when \ } k\to\infty {\rm \ and \ }  \norm{f_{n_k}\chi_{\{f_{n_k}>k\}}}_{p,1}\geq 1/2  
{\rm \ for \ each \ } k.
$$
Thus we can select another subsequence $(n_{k_i})$ such that 
$$
m(\{f_{n_{k_i}}>k_i\})<\frac{1}{2^i} {\rm \ for \ each \ } i.
$$
Define a new sequence $(g_i)$ by 
$$
g_i(t)=[f_{n_{k_i}}\chi_{\{f_{n_{k_i}}>{k_i}\}}]^*(t-1/2^i) {\rm \ for \ } t\geq 1/2^i
$$
and $g_i(t)=0$ otherwise. Notice that $g_i$ is equimeasurable with $f_{n_{k_i}}\chi_{\{f_{n_{k_i}}>{k_i}\}}$  
for each $i$. 
Moreover,  $(g_i)$ is disjoint sequence and all supports of $(g_i)$ fit into $[0,1]$, thus we may regard 
$(g_i)$ as a subset of  $L^{p,1}[0,1]$. We have  $1/2\leq \norm{g_i}_{p,1}\leq 1$ for each $i$. 
However, the space $L^{p,1}[0,1]$ is $1$-disjointly homogeneous (cf. \cite{FJT75} and Subsection \ref{2.4}) 
and so there is subsequence of $(g_i)$ which is equivalent to the $l^1$ basis. Without loss of generality 
we may assume that $(g_i)$ itself is equivalent to the $l^1$ basis. Then there is $\eta>0$ such that for 
each $a=(a_i)\in l^1$ we have 
$$
\eta \sum_{i=1}^{\infty}|a_i|\leq \norm{\sum_{k=1}^{\infty}a_ig_i}_{L^{p,1}[0,1]}
\leq \norm{\sum_{k=1}^{\infty}a_if_{n_{k_i}}\chi_{\{f_{n_{k_i}}>{k_i}\}}}_{E}
\leq \norm{\sum_{k=1}^{\infty}a_if_{n_{k_i}}}_{E}\leq \sum_{i=1}^{\infty}|a_i|.
$$
Thus it remains to prove the claim. Suppose (\ref{star}) does not hold. This means there is $k_0$ 
such that for each $n$ 
\begin{equation}\label{niestar}
\norm{f_n\chi_{\{f_n>k_0\}}}_{p,1}< 1/2,
\end{equation}
which implies 
\begin{equation}
\norm{f_n\chi_{\{f_n\leq k_0\}}}_{p,1}\geq 1/2,
\end{equation}
for each $n$. Then
$$
1/2 \leq \norm{f_n\chi_{\{f_n\leq k_0\}}}_{p,1}\leq \int_0^{k_0}d_{f_n}^{1/p}(t)dt
\leq \left(\int_0^{k_0}d_{f_n}(t)dt\right)^{1/p}k_0^{1/p'}.
$$
This is
$$
\norm{f_n}_1\geq \frac{1}{2^pk_0^{p/p'}},
$$
which contradicts $\norm{f_n}_1\to 0$ and the claim is proved. 
\endproof

Astashkin, Kalton and Sukochev \cite{AKS08} proved that order continuous Orlicz space $L^{\varphi}[0,1]$ 
satisfies the Dunford--Pettis criterion if and only if complementary function $\varphi*$ satisfies $\Delta_0$ 
condition, that is,
\begin{equation} \label{4.4}
\lim_{u \rightarrow \infty} \frac{\varphi^*(\lambda u)}{\varphi^*(u)} = \infty ~{\rm for ~ some} ~\lambda > 1.
\end{equation}

In case of Orlicz spaces on $I=[0,\infty)$ it is in order to ask if the analogous theorem may be proved 
with (\ref{2}) condition on $\varphi^*$ extended to small arguments. This is, if, analogously as for the 
finite measure space, the condition
$$
\lim\limits_{t\rightarrow 0}\frac{\varphi^*(\lambda t)}{\varphi^*(t)}=\infty {\rm \ for \ some \ }0<\lambda<1
$$
together with (\ref{4.4}) is sufficient or necessary for $L^{\varphi}[0,\infty)$ to satisfy the Dunford--Pettis 
criterion. However, it appears, that this is not the case and we have the following characterization of 
Orlicz spaces on $[0,\infty)$ satisfying the Dunford--Pettis criterion. Notice that already another 
characterization of 1-DH Orlicz spaces on $I=[0,\infty)$ was given in \cite[Theorem 5.1]{FHSTT14}. 
It was, however, formulated in terms of $C_{\varphi}$ set, which is much more difficult to understand 
than simply saying that a respective function satisfies the condition $\Delta_0$. Moreover, the proof 
of \cite[Theorem 5.1]{FHSTT14} relies on the general construction of \cite{NL75}, while our is quite 
elementary. 

\begin{theorem}
The Orlicz space $L^{\varphi}[0,\infty)$ is $1$-disjointly homogeneous if and only if $\varphi\in\Delta_2$, 
$\varphi^*\in \Delta_0$  and $\varphi$ is equivalent to the linear function for small arguments, i.e. there are 
constants $C, c, d, u_0 > 0$ such that for each $0 < u < u_0$ there holds
$$
c u \leq \varphi(d u) \leq Cu.
$$ 

\end{theorem}


\proof Sufficiency follows by an argument analogous to the one from previous theorem, since under our 
assumptions 
$$
L^{\varphi}[0,\infty) = (L^{\varphi}\cap L^1)[0,\infty).
$$
The only difference appears when proving the claim (\ref{star}) from the previous proof, thus we will provide 
a detailed argument only for the claim: if $(f_n)\subset  L^{\varphi}[0,\infty)$ is a sequence of disjoint positive 
functions such that $\|f_n\|_1\to0$ and  $\|f_n\|_{\varphi}=1$, then for each $k\in \mathbb{N}$ there is 
$n\in \mathbb{N}$ such that
\begin{equation}\label{starfi}
\norm{f_n\chi_{\{f_n>k\}}}_{\varphi}> 1/2.
\end{equation}
Assume it is not the case, this means there is $k_0$ such that for each $n$ 
\begin{equation}\label{niestarfi}
\norm{f_n\chi_{\{f_n>k_0\}}}_{\varphi}\leq 1/2,
\end{equation}
which implies 
\begin{equation}
\norm{f_n\chi_{\{f_n\leq k_0\}}}_{\varphi}> 1/2.
\end{equation}
The latter means that 
$$
\int_{\{f_n\leq k_0\}}\varphi (2f_n(t))dt>1.
$$
Notice that convexity of $\varphi$ implies
$$
\varphi(2t)\leq \frac{\varphi(2k_0)}{k_0}t {\rm \ for\ each\ }0<t\leq k_0.
$$
Consequently,
$$
1< \int_{\{f_n\leq k_0\}}\varphi (2f_n(t))dt\leq \frac{\varphi(2k_0)}{k_0}\int_{\{f_n\leq k_0\}}f_n(t)dt,
$$
this is 
$$
\|f_n\|_1\geq \frac{k_0}{\varphi(2k_0)},
$$
for each $n$ and we arrived to contradiction with $\|f_n\|_1\to0$.

Necessity of $\varphi_*\in \Delta_0$ is evident by the Astashkin, Kalton and Sukochev \cite[Proposition 5.8]{AKS08}. 
Also, each $1$-DH space is separable, so  $\varphi\in\Delta_2$. To see that $\varphi$ has to be equivalent 
to the linear function for small arguments it is enough to consider the following sequence
$$
f_n=\chi_{[n-1, n]}, {\rm \ where\ } n \in \mathbb N.
$$
It is evident that each subsequence of $(f_n)$ is equivalent (even equal) to the unit basis of Orlicz 
sequence space $l^{\varphi}$. Thus $l^{\varphi}$ has to be equal to $l^1$, which happen if and only if 
$\varphi$ is equivalent to the identity function for small arguments.
\endproof

\section{The $\Delta_0$-condition and examples}

Let us recall that an Orlicz function $\varphi$ satisfies the {\it $\Delta_0$-condition}  if for some 
$\lambda > 1$ 
\begin{equation} \label{delta0}
\lim_{u \rightarrow \infty} \frac{\varphi (\lambda u)}{\varphi (u)} = \infty.
\end{equation}
Since it played the crucial role in the previous section, let us discuss this  condition. 
First of all notice, that contrary to the $\Delta_2^{\infty}$-condition, 
in the definition of $\Delta_0$-condition one cannot equivalently assume that (\ref{delta0}) holds for all 
$\lambda>1$, as can be seen in \cite[Example 7]{LLQR08}. Furthermore, the $\Delta_0$-condition may 
be seen as a strong negation of the $ \Delta_2^{\infty}$-condition. Going a slightly deeper into this 
analogy, if $\varphi\not \in\Delta_2^{\infty}$ then $L^{\varphi}$ contains a copy of $l^{\infty}$ based 
on a sequence of disjointly supported functions, while if $\varphi \in\Delta_0$ then each normalized 
sequence of disjointly supported functions contains a subsequence that spans $l^{\infty}$ (and this 
is how $\Delta_0$-condition connects with 1-DH spaces by duality). 

In addition to the previously mentioned papers, condition $\varphi^* \in \Delta_0$ appears 
in the papers \cite{Al98}, \cite{AS18}, \cite{La08} and \cite{St19}.

The so-called Matuszewska--Orlicz indices of Orlicz functions are helpful in presenting specific 
examples of Orlicz's functions with the $\Delta_0$-condition. The {\it lower and upper 
Matuszewska--Orlicz indices}  (for large arguments or at infinity) 
$\alpha_{\varphi}^{\infty}, \beta_{\varphi}^{\infty}$ of an Orlicz function $\varphi$ are 
defined by
$$
\alpha_{\varphi}^{\infty} = \lim\limits_{t\rightarrow 0^+}\frac{\ln{M(t,\varphi)}}{\ln{t}},\ 
\beta_{\varphi}^{\infty}=\lim\limits_{t\rightarrow \infty}\frac{\ln{M(t,\varphi)}}{\ln{t}},
$$
where $M_{\varphi}^{\infty}(t) = \limsup\limits_{u\rightarrow \infty}\frac{\varphi(tu)}{\varphi(u)}$. 
The basic properties of these indices are: 
$1 \leq \alpha_{\varphi}^{\infty} \leq \beta_{\varphi}^{\infty} \leq \infty$, $\beta_{\varphi}^{\infty}< \infty$ 
if and only if $\varphi \in \Delta_2^{\infty}$ and
\begin{equation} \label{indeksy}
\frac{1}{\alpha_{\varphi}^{\infty}} + \frac{1}{\beta_{\varphi^*}^{\infty}} = \frac{1}{\alpha_{\varphi^*}^{\infty}} 
+ \frac{1}{\beta_{\varphi}^{\infty}} = 1.
\end{equation} 
More information on indices can be found in \cite{KPS82}, \cite{Ma85} and \cite{Ma89}.
\vspace{2mm}

\begin{proposition}\label{delta0} 
If $\varphi \in \Delta_0$, then $\alpha_{\varphi}^{\infty} = \beta_{\varphi}^{\infty} = \infty$ and  
$\alpha_{\varphi^*}^{\infty} = \beta_{\varphi^*}^{\infty} =1$.
\end{proposition}

\begin{proof}
If $\varphi \in \Delta_0$, then $\varphi \notin \Delta_2^{\infty}$ and $\beta_{\varphi}^{\infty} = \infty$. 
Hence, by (\ref{indeksy}), $\alpha_{\varphi^*}^{\infty} = 1$. On the other hand, if $\varphi \in \Delta_0$, 
then for any $\eta > \lambda$ we have
\begin{eqnarray*} \label{delta0}
\liminf_{u \rightarrow \infty} \frac{\varphi (\eta u)}{\varphi (u)} &\geq&
\liminf_{u \rightarrow \infty} \frac{\varphi (\lambda u)}{\varphi (u)} 
= \lim_{u \rightarrow \infty} \frac{\varphi (\lambda u)}{\varphi (u)} = \infty,
\end{eqnarray*}
and so $\alpha_{\varphi}^{\infty} = \infty$, from which by (\ref{indeksy}) we get 
$\beta_{\varphi^*}^{\infty} =1$. In particular, $\varphi^* \in \Delta_2^{\infty}$.
\end{proof}

We do not know any example of an Orlicz 
function $\varphi$ with $\alpha_{\varphi}=\beta_{\varphi} = \infty$ and 
$\varphi \not\in \Delta_0$. Let us therefore state the question:

\begin{question}
Does $\alpha_{\varphi}=\beta_{\varphi}=\infty$ imply $\varphi\in \Delta_0$?
\end{question}

Nevertheless, using Simonenko indices instead of Matuszewska--Orlicz we can formulate sufficient 
condition for $\varphi \in \Delta_0$. 

\begin{theorem}\label{ind}
If an Orlicz function $\varphi$ satisfies
\begin{equation}\label{ind1}
\lim_{u\to \infty}\frac{u\varphi'(u)}{\varphi(u)}=1,
\end{equation}
then $\lim\limits_{u \to \infty} \frac{\varphi^* (\lambda u)}{\varphi^* (u)} = \infty$ for any $\lambda > 1$. 
In particular, $\varphi^* \in \Delta_0$.
\end{theorem}
\proof
First of all, let's see that if $\lim_{u\to \infty}\frac{u\varphi'(u)}{\varphi(u)} = \infty$, then
\begin{equation}\label{regular}
\lim\limits_{u \to \infty} \frac{\varphi (\lambda u)}{\varphi(u)} = \infty ~{\rm for ~ any} ~\lambda > 1.
\end{equation}
Really, since $\psi(t):= \frac{t \varphi'(t)}{\varphi(t)} > c$ for any large $c$ and large enough $t$, 
it follows for arbitrary $\lambda > 1$ that
$$
\frac{\varphi(\lambda u)}{\varphi(u)} = \exp (\int_u^{\lambda u} \frac{\psi(t)}{t} \, dt)
\geq \exp (\int_u^{\lambda u} \frac{c}{t} \, dt) = \lambda^c 
$$
which gives (\ref{regular}).

Now, we are ready to prove Theorem \ref{ind}. Define for $t,u>0$
\[
p(t)=\frac{\varphi(t)}{t}\ \ {\rm \ and\ } \ \ \varphi_1(u)=\int_0^up(t)dt.
\]
Then, since $p$ is non-decreasing, for each $\lambda>1$ there holds
\[
1\geq \lim_{u\to \infty}\frac{\varphi_1(u)}{up(u)}\geq \lim_{u\to \infty}\frac{\int_{u/\lambda}^u p(t)dt}{up(u)}
\geq \lim_{u\to \infty}\frac{p(u/\lambda)(1-1/\lambda)}{p(u)}=1-1/\lambda,
\]
which means that 
\begin{equation}\label{ind2}
\lim_{u\to \infty}\frac{\varphi_1(u)}{up(u)}=1. 
\end{equation}
Let $v>0$ and $u=p^{-1}(v)$ (where $p^{-1}$ denotes the right continuous inverse of $p$), then 
\[
uv=up(u)=\varphi_1(u)+\varphi_1^*(v),
\]
i.e. 
\[
\frac{\varphi_1(u)}{up(u)}+\frac{\varphi_1^*(v)}{vp^{-1}(v)}=1.
\]
Thus (\ref{ind2}) implies 
\[
\lim_{v\to \infty}\frac{vp^{-1}(v)}{\varphi_1^*(v)}=\infty.
\]
According to (\ref{regular}) it means that 
\begin{equation}\label{ind3}
\lim_{u\to \infty}\frac{\varphi_1^*(\lambda u)}{\varphi_1^*(u)}=\infty,
\end{equation}
for each $\lambda>1$. 
But, once again, monotonicity of $p$ gives $\varphi_1(u)\leq \varphi(u)$ for each $u>0$, which, in turn, implies 
 $\varphi_1^*(u)\geq \varphi^*(u)$ for each $u>0$.
On the other hand, for each $\eta>1$ and $u>0$ we have 
\[
\varphi_1(\eta u)\geq \int_{u}^{\eta u}p(t)dt\geq u p(u)(\eta-1)=\varphi(u)(\eta-1).
\]
Thus, making respective substitution in the definition of conjugate functions, we conclude that also
 $\frac{1}{\eta-1}\varphi_1^*(\frac{\eta-1}{\eta}u)\leq \varphi^*(u)$ for each $u>0$.
Finally, let $\lambda>1$ and choose $\eta>1$ in such a way, that $\lambda \frac{\eta-1}{\eta}>1$. 
We have finally by (\ref{ind3})
\[
\lim_{u\to \infty}\frac{\varphi^*(\lambda u)}{\varphi^*(u)}\geq 
\lim_{u\to \infty}\frac{\frac{1}{\eta-1}\varphi_1^*(\lambda \frac{\eta-1}{\eta} u)}{\varphi_1^*(u)}=\infty. 
\]
\endproof

The {\it lower} and {\it upper Simonenko indices} (at infinity) $a_{\varphi}^{\infty}, b_{\varphi}^{\infty}$ of 
an Orlicz function $\varphi$ are defined by
\[
a_{\varphi}^{\infty} = \liminf_{u \to \infty} \frac{u \varphi^{\prime}(u)}{\varphi(u)}, \, b_{\varphi}^{\infty} = \limsup_{u \to \infty} \frac{u \varphi^{\prime}(u)}{\varphi(u)}.
\]
The basic properties are (see \cite{KR61}, \cite{Ma85}, \cite{Ma89}, \cite{RR91}): 
$1 \leq a_{\varphi}^{\infty} \leq \alpha_{\varphi}^{\infty} \leq \beta_{\varphi}^{\infty} \leq b_{\varphi}^{\infty} \leq \infty$, 
$b_{\varphi}^{\infty} < \infty$ if and only if $\varphi \in \Delta_2^{\infty}$, and for an N-function $\varphi$ we have
\begin{equation}\label{Simo}
\frac{1}{a_{\varphi}^{\infty}} + \frac{1}{b_{\varphi^*}^{\infty}} = \frac{1}{a_{\varphi^*}^{\infty}} + \frac{1}{b_{\varphi}^{\infty}} = 1.
\end{equation}

\begin{corollary}
If for an N-function $\varphi$ we have $a_{\varphi}^{\infty}=b_{\varphi}^{\infty}=\infty$, then $\varphi \in \Delta_0$. 
\end{corollary}
\proof
Using relations (\ref{Simo}) we obtain condition (\ref{ind1}) for $\varphi^*$ and by the fact that conjugation 
$\varphi\mapsto \varphi^*$ is an involution, we conclude that $\varphi \in \Delta_0$.
\endproof

Notice that $a_{\varphi}^{\infty}=b_{\varphi}^{\infty}=\infty$ cannot be necessary for $\varphi \in \Delta_0$ because 
of Theorem \ref{ind} and Example 7 in \cite{LLQR08}.

\begin{example}
The N-function 
\begin{equation*}
~~\varphi (u) = 
\begin{cases}
\frac{u^2}{2} & ~\mathrm{if} ~0 \leq u \leq 1, \\ 
u \ln u + \frac{1}{2} & ~\mathrm{if} ~u \geq 1,
\end{cases}
\end{equation*}
satisfies the $\Delta_2$-condition (even for all $u > 0$) and its 
complementary function
\begin{equation*}
~~\varphi^* (u) = 
\begin{cases}
\frac{u^2}{2} & ~\mathrm{if} ~0 \leq u \leq 1, \\ 
e^{u-1} - \frac{1}{2} & ~\mathrm{if} ~u \geq 1,
\end{cases}
\end{equation*}
satisfies the $\Delta_0$-condition.
\end{example}

 In many cases it is impossible to find explicit formula for the complementary N-function, but
using Theorem \ref{ind} we can easily see why $\varphi^* \in \Delta_0$.

\begin{example}
For Orlicz functions $\varphi_r (u) = u \ln^r(1+u), r > 0$, $\varphi_a (u) = u \sqrt{1 + a \ln(1+u)}, a > 0$, 
and $\varphi_b (u) = u \exp(\sqrt{1 + a \ln^+u}), b > 0$ we have 
\[
\frac{u \varphi_r ^{\prime}(u)}{\varphi_r (u)} = 1 + \frac{r u}{(1+u) \ln(1+u)} \to 1 ~{\rm as} ~ u \to \infty, 
\]
\[
\frac{u \varphi_a ^{\prime}(u)}{\varphi_a (u)} = 1 + \frac{a u}{2(1+u) [1+ a \ln(1+u)]} \to 1 ~{\rm as} ~ u \to \infty,
\]
and 
\[
\frac{u \varphi_b ^{\prime}(u)}{\varphi_b (u)} = 1 + \frac{a}{2\sqrt{1 + a \ln^+u}} \to 1 ~{\rm as} ~ u \to \infty.
\]
From Theorem \ref{ind} we obtain that $\varphi_r^*, \varphi_a^*, \varphi_b^* \in \Delta_0$. 
\end{example}

\section{Weak compactness of pointwise multipliers}

Finally we are in a position to characterize weakly compact pointwise multipliers between Banach 
function lattices. 

Given  two Banach function lattices $X$ and $Y$ on $\Omega$ we define the space of pointwise 
multipliers from $X$ to $Y$  as
\[
M(X, Y)=\{f\in L^0(\Omega) \colon fg\in Y\ {\rm for\ all\ } g\in X\},
\]
with the operator norm 
\[
\norm{f}_{M}=\sup_{\norm{g}_X\leq1}\norm{fg}_Y. 
\]
If $Y$ has the Fatou property, then the space $M(X, Y)$ has the Fatou property (see \cite{MP89}).
Of course, each $f\in M(X, Y)$ defines a pointwise multiplication operator 
$
M_f \colon X\rightarrow Y
$ 
by 
$$
M_f(g) = f g,\ \  g\in X
$$
and $\norm{M_f}=\norm{f}_{M(X, Y)}$.

\begin{lemma}
\textit{Let $X$ and $Y$ be Banach function lattices on $\Omega$. If $f\in (M(X, Y))_a$, then 
$M_f \colon X\to Y$ is a weakly compact operator.}
\end{lemma}

\begin{proof}
We will show that $M_f B(X)$ is $F$-equi-integrable set. Let $(A_n)$ be a sequence of measurable sets 
such that $A_n\downarrow\emptyset$. We have
\begin{align*}
\sup\limits_{h\in M_f B(X)}\norm{h\chi_{A_n}}_X=\sup\limits_{g\in B(X)}\norm{fg\chi_{A_n}}_X 
= \norm{f\chi_{A_n}}_M\rightarrow 0,
\end{align*}
as $n\rightarrow\infty$.
\end{proof}

In the case $Y$ satisfies the Dunford--Pettis criterion, we see that order continuity of $f \in M(X, Y)$ 
is also necessary for weak compactness of $M_f$. 

\begin{theorem}
\textit{Let $X$ and $Y$ be Banach function lattices on $\Omega$ such that $Y$ is 1-DH 
and let $f\in M(X, Y)$. Then $M_f \colon X\to Y$ is weakly compact if and only if $f \in (M(X, Y))_a$.}
\end{theorem}

\begin{proof}
Since $M_f$ is weakly compact then $M_f B(X)$ is relatively weakly compact subset of $Y$. From 
Theorem \ref{charakteryzacja} it follows that  $M_f B(X)$ is $F$-equi-integrable set. Thus for arbitrary 
sequence of measurable sets $(A_n)$ such that $A_n\downarrow\emptyset$ we have 
$$
\norm{f\chi_{A_n}}_M=\sup\limits_{g\in B(X)}\norm{fg\chi_{A_n}}_Y\rightarrow 0 \text{ when }n\rightarrow\infty,
$$
which means that $f\in (M(X, Y))_a$.
\end{proof}

Finally we present one immediate consequence of the previous theorem. Notice that there are in general 
no results describing properties of the space $M(X, Y)$ in terms of properties of spaces $X, Y$. 
It is just because properties of $M(X, Y)$ depend rather on coincidence of $X$ and $Y$, than on their 
properties separately. From this point of view the following conclusion seems surprising. 

\begin{corollary}
\textit{Let $X$ and $Y$ be Banach function lattices over $I$. If $X$ is a reflexive space and $Y$ is 
1-DH, then $M(X, Y)\in (OC)$.}
\end{corollary}

\subsection*{Acknowledgements} 
The authors would like to thank Professor Sergey V. Astashkin for consultations on the subject and his very valuable comments on the first version of the paper.
The first and the third author were
supported by the National Science Center (Narodowe Centrum Nauki), Poland (project
no. 2017/26/D/ST1/00060).


\end{document}